\documentclass[12pt]{amsart}
\usepackage{graphicx}
\usepackage[headings]{fullpage}
\usepackage{amssymb,epic,eepic,epsfig,amsbsy,amsmath,amscd}
\numberwithin{equation}{section}
                        \theoremstyle{plain}
\usepackage{mathrsfs}

\newcommand\no[1]{}

\newtheorem{theorem}{Theorem}[section]

\newtheorem{lemma}[theorem]{Lemma}
\newtheorem{corollary}[theorem]{Corollary}
\newtheorem{proposition}[theorem]{Proposition}

\theoremstyle{definition}
\newtheorem{remark}[theorem]{Remark}

\def\BC{\mathbb C}

\def\BZ{\mathbb Z}

\def\BT{\mathbb T}

\def\CM{\mathcal M}

\def\CT{\mathcal T}

\def\la{\langle}
\def\ra{\rangle}

\DeclareMathOperator{\tr}{\mathrm tr}

\def\be { \begin{equation} }
\def\ee { \end{equation} }

\begin{document}

\title[ATAP of genus one two-bridge knots]{Adjoint Twisted Alexander polynomials of genus one two-bridge knots}

\author[Anh T. Tran]{Anh T. Tran}
\address{Department of Mathematical Sciences, The University of Texas at Dallas, Richardson, TX 75080, USA}
\email{att140830@utdallas.edu}

\begin{abstract}
We give explicit formulas for the adjoint twisted Alexander polynomial and the nonabelian Reidemeister torsion of genus one two-bridge knots.
\end{abstract}

\thanks{2010 {\em Mathematics Classification:} Primary 57N10. Secondary 57M25.}

\maketitle

\section{Introduction}

The twisted Alexander polynomial, a generalization of the Alexander polynomial \cite{Al}, was introduced by Lin \cite{Li} for knots in the 3-sphere and by Wada \cite{Wa} for finitely presented groups. Twisted Alexander polynomials have been extensively studied in recent years, see the survey papers \cite{FV, Mo} and references therein. The adjoint action, $\text{Ad}$, is the conjugation on the Lie algebra $sl_2(\BC)$ by the Lie group $SL_2(\BC)$. Suppose $K \subset S^3$ is a knot and $\pi_1(K)$ its knot group. For each representation $\rho$ of $\pi_1(K)$ into $SL_2(\BC)$, the composition $\text{Ad} \circ \rho$ is a representation of $\pi_1(K)$ into $SL_3(\BC)$ and hence, by \cite{Wa}, one can define a rational function $\Delta^{\text{Ad} \circ \rho}_{K}(t)$, called the adjoint twisted Alexander polynomial associated to $\rho$. The polynomial $\Delta^{\text{Ad} \circ \rho}_{K}(t)$ was calculated for the figure eight knot by Dubois and Yamaguchi \cite{DY}, and for torus knots and twist knots by the author \cite{Tr2}. In this paper we compute the adjoint twisted Alexander polynomial for genus one two-bridge knots, a class of two-bridge knots which includes twist knots.

Let $J(k,l)$ be the knot/link in Figure 1, where $k,l$ denote 
the numbers of half twists in the boxes. Positive (resp. negative) numbers correspond 
to right-handed (resp. left-handed) twists. 
Note that $J(k,l)$ is a knot if and only if $kl$ is even. It is known that the set of all genus one two-bridge knots  is the same as the set of all the knots $J(2m,2n)$ with $mn \not= 0$, see e.g. \cite[page 203]{BZ}. The knots $J(2,2n)$ are known as twist knots. For more information on $J(k,l)$, see \cite{HS}.

\begin{figure}[th]
\centerline{\psfig{file=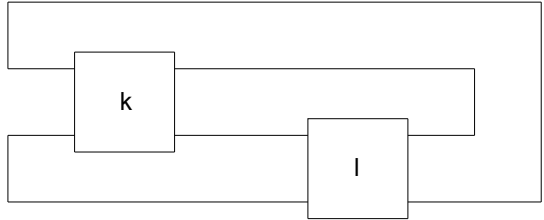,width=3.5in}}
\vspace*{8pt}
\caption{The knot/link $J(k,l)$. }
\end{figure} 

From now on we fix $K=J(2m,2n)$ with $mn \not=0$. The knot group of $K$ has a presentation $\pi_1(K)= \la a, b \mid w^na=bw^n \ra$ where $a,b$ are meridians and $w=(ba^{-1})^m(b^{-1}a)^m$. A representation $\rho:  \pi_1(K) \to SL_2(\BC)$ is called nonabelian if 
the image of $\rho$ is a nonabelian subgroup of $SL_2(\BC)$. Suppose $\rho: \pi_1(K) \to SL_2(\BC)$ is a nonabelian representation. Up to conjugation, we may assume that $$\rho(a) = \left[ \begin{array}{cc}
s & 1 \\
0 & s^{-1} \end{array} \right] \quad \text{and} \quad \rho(b) = \left[ \begin{array}{cc}
s & 0 \\
2-y & s^{-1} \end{array} \right]$$ where $s \not=0$ and $y \not= 2$ satisfy the Riley equation $\phi_K(s,y)=0$, see \cite{Ri, Le}. Note that $y=\tr \rho(ab^{-1})$.  The polynomial $\phi_K(s,y)$ is explicitly given in Proposition \ref{Riley} by
$$\phi_K(s,y) = S_{n-2}(z) - \big[ 1  - (y+2-x^2)S_{m-1}(y) \big( S_{m-1}(y)- S_{m-2}(y) \big) \big] S_{n-1}(z),$$
where $x :=\tr \rho(a)=s + s^{-1}$ and $$z:= \tr \rho(w)=2S^2_{m}(y) - 2yS_m(y)S_{m-1}(y) + \big(  (2x^2-2)(1-y) + y^2 \big) S^2_{m-1}(y).$$Here $S_k(v)$'s are the Chebychev polynomials of the second kind defined by $S_0(v)=1$, $S_1(v)=v$ and $S_{k}(v) = v S_{k-1}(v) - S_{k-2}(v)$ for all integers $k$.

 Let
\begin{eqnarray*}
A &=& 2n \Big( \frac{2m +2 S_m(y)S_{m-1}(y) - y S^2_{m-1}(y)}{y+2} \Big),\\
B &=& -4 \big( S_m(y) + S_{m-1}(y) \big)^2  \Big( \frac{m S_{m}(y) + (m+1) S_{m-1}(y)}{(y+2) S_{m-1}(y)} \Big)\\
&& + \, 6m S^2_m(y)+(4+4m-4n-2my) S_{m}(y) S_{m-1}(y)\\
&& + \,  (2+2m-4mn-y+2ny+2mny) S^2_{m-1}(y) ,\\
C &=& - (2mn-1)  \big( 2S_m(y) - yS_{m-1} (y) \big)^2.
\end{eqnarray*}
The adjoint twisted Alexander polynomial of $K=J(2m,2n)$ is computed as follows.

\begin{theorem} \label{thm-atp}
Suppose $\rho: \pi_1(K) \to SL_2(\BC)$ is a nonabelian representation. We have
\begin{eqnarray*}
\Delta_{K}^{Ad \circ \rho}(t) &=& \frac{t-1}{(y+2-x^2) \big( 4 - x^2  + (y-2)(y+2-x^2) S^2_{m-1}(y) \big)} \\
&& \times \left( mn t^2 - \frac{Ax^4 + Bx^2+C}{4   + (y-2)(y+2-x^2) S^2_{m-1}(y)} \, t + mn\right).
\end{eqnarray*}
\end{theorem}

Theorem \ref{thm-atp} generalizes the formula for the adjoint twisted Alexander polynomial of the twist knot $K=J(2,2n)$ in \cite{Tr2}.

It is known that $\Delta^{\text{Ad} \circ \rho}_{K}(t)$ coincides with the nonabelian Reidemeister torsion polynomial $\CT^{\rho}_{K}(t)$ \cite{Ki, KL}. As a consequence of this identification, one can calculate the nonabelian Reidemeister torsion $\BT^{\rho}_{K}$ for any longitude-regular $SL_2(\BC)$-representation $\rho$ of the knot group of $K$ by applying the following formula of Yamaguchi:
\begin{equation} \label{ya}
\BT^{\rho}_{K}=-\lim_{t \to 1} \frac{\CT^{\rho}_{K}(t)}{t-1},
\end{equation} 
see \cite{Ya}. We refer the reader to \cite{Po1, Po2, Du, DHY, DY} for definitions of $\CT^{\rho}_{K}(t)$ and $\BT^{\rho}_{K}$. 

From Theorem \ref{thm-atp} and Formula \eqref{ya}, we have the following.

\begin{corollary} 
Suppose $\rho: \pi_1(K) \to SL_2(\BC)$ is a longitude-regular nonabelian representation. We have
\begin{eqnarray*}
\BT_{K}^{\rho}&=& \frac{1}{(y+2-x^2) \big( 4 - x^2  + (y-2)(y+2-x^2) S^2_{m-1}(y) \big)} \\
&& \times \left( 2mn - \frac{Ax^4 + Bx^2+C}{4   + (y-2)(y+2-x^2) S^2_{m-1}(y)} \right).
\end{eqnarray*}
\end{corollary}

\begin{remark}
(1) The nonabelian Reidemeister torsion $\BT_K^{\rho}$ appears in the volume conjecture of Kashaev and Murakami-Murakami \cite{Ka, MM} which relates the colored Jones polynomial of a knot $K \subset S^3$ and its hyperbolic volume. 

(2) A combinatorial formula for the nonabelian Reidemeister torsion of a hyperbolic two-bridge knot associated to the holonomy representation has recently been given by Ohtsuki and Takata \cite{OT}.
\end{remark}

The paper is organized as follows. In Section \ref{atp} we briefly recall the definition of the adjoint twisted Alexander polynomial. In Section \ref{matrix} we present matrix computations which are needed in the proof of Theorem \ref{thm-atp}. Finally, we prove Theorem \ref{thm-atp} in Section \ref{proof}.

\section{Adjoint twisted Alexander polynomial}
\label{atp}

\subsection{Twisted Alexander polynomial} Let $K$ be a knot and $\pi_1(K)=\pi_1(S^3\backslash K)$ its knot group. We fix a presentation
$$
\pi_1(K)=
\langle a_1,\ldots,a_\ell~|~r_1,\ldots,r_{\ell-1}\rangle.
$$
(This might not be a Wirtinger representation, but must be of deficiency one.) 

Let $f:\pi_1(K)\to H_1(S^3\backslash K,\BZ)
\cong {\BZ}
=\langle t
\rangle$ 
be the abelianization homomorphism and $\rho:\pi_1(K)\to SL_k(\BC)$ a representation. These maps naturally induce two ring homomorphisms $\widetilde{f}:{\BZ}[\pi_1(K)]\rightarrow {\BZ}[t^{\pm1}]$ and $\widetilde{\rho}: {\BZ}[\pi_1(K)] \rightarrow \CM(k,{\BC})$, 
where ${\BZ}[\pi_1(K)]$ is the group ring of $\pi_1(K)$ 
and 
$\CM(k,{\BC})$ is the matrix algebra of degree $k$ over ${\BC}$. 
Then 
$\widetilde{\rho}\otimes\widetilde{f}: {\BZ}[\pi_1(K)]\to \CM\left(k,{\BC}[t^{\pm1}]\right)$ 
is a ring homomorphism. 
Let 
$F_\ell$ be the free group on 
generators $a_1,\ldots,a_\ell$ and 
$\Phi:{\BZ}[F_\ell]\to \CM\left(k,{\BC}[t^{\pm1}]\right)$
the composition of the surjective map 
${\BZ}[F_\ell]\to{\BZ}[\pi_1(K)]$ 
induced by the presentation of $\pi_1(K)$ 
and the map 
$\widetilde{\rho}\otimes\widetilde{f}:{\BZ}[\pi_1(K)]\to \CM(k,{\BC}[t^{\pm1}])$. 

We consider the $(\ell-1)\times \ell$ matrix $M$ 
whose $(i,j)$-component is the $k\times k$ matrix 
$$
\Phi\left(\frac{\partial r_i}{\partial a_j}\right)
\in \CM\left(k,{\BC}[t^{\pm1}]\right),
$$
where 
$\frac{\partial}{\partial a}$ 
denotes the Fox derivative. 
For 
$1\leq j\leq \ell$, 
let $M_j$ be
the $(\ell-1)\times(\ell-1)$ matrix obtained from $M$ 
by removing the $j$th column. 
We regard $M_j$ as 
a $k(\ell-1)\times k(\ell-1)$ matrix with coefficients in 
${\BC}[t^{\pm1}]$. 
Then Wada's twisted Alexander polynomial 
of the knot $K$ associated to the representation $\rho:\pi_1(K)\to SL_k({\BC})$ 
is defined to be the rational function 
$$
\Delta^{\rho}_{K}(t)
=\frac{\det M_j}{\det\Phi(1-a_j)}. 
$$
It is defined 
up to a factor $t^{km}~(m\in{\BZ})$, see \cite{Wa}. 

\subsection{Adjoint twisted Alexander polynomial} The adjoint action, $\text{Ad}$, is the conjugation on the Lie algebra $sl_2(\BC)$ by the Lie group $SL_2(\BC)$. For $A \in SL_2(\BC)$ and $g \in sl_2(\BC)$ we have $\text{Ad}_A(g)=AgA^{-1}$. For each representation $\rho: \pi_1(K) \to SL_2(\BC)$, the composition $\text{Ad} \circ \rho: \pi_1(K) \to SL_3(\BC)$ is a representation and hence  
one can define the twisted Alexander polynomial $\Delta^{\text{Ad} \circ \rho}_{K}(t)$. We call $\Delta^{\text{Ad} \circ \rho}_{K}(t)$ the adjoint twisted Alexander polynomial associated to $\rho$. In this paper we are interested in the adjoint twisted Alexander polynomial associated to nonabelian $SL_2(\BC)$-representations.

\section{Matrix Computations}

\label{matrix}

\begin{proposition} \label{prop1}
Suppose $M = \left[ \begin{array}{cc}
e & f \\
g & h \end{array} \right] \in SL_2(\BC).$ 
Let $\mu = \tr M$, $X = S^2_{n-1}(\mu)$ and $Y = S_{n-1}(\mu) S_{n-2}(\mu)$. Then 
$$
\sum_{i=0}^{n-1} (Ad_M)^i = \frac{1}{\mu^2-4} \left[ \begin{array}{ccc}
C_{11} & C_{12} & C_{13} \\
C_{21} & C_{22} & C_{23} \\
C_{31} & C_{32} & C_{33}\end{array} \right],
$$
where 
\begin{eqnarray*}
C_{11} &=&  2nfg + h^2 (2X-\mu Y) -  2h(\mu X -  2 Y) + (\mu^2-2)X - \mu Y,\\
C_{12} &=& 2f \big( n(e-h) + h (2X -\mu Y) - \mu X + 2 Y \big),\\
C_{13} &=& f^2 (2n- 2 X + \mu Y),\\
C_{21} &=& -g \big( n(h-e) - h (2X -\mu Y) + \mu X - 2 Y \big),\\
C_{22} &=& n(e-h)^2+2 fg(2X - \mu Y),\\
C_{23} &=& f (n(e-h) + h \big( 2X - \mu Y) - \mu X + (\mu^2-2)Y \big),\\
C_{31} &=& g^2 (2n- 2X + \mu Y),\\
C_{32} &=& -2g \big( n(h-e)-h(2X-\mu Y) + \mu X - (\mu^2-2)Y \big),\\
C_{33} &=& 2n fg + h^2(2X - \mu Y) - 2 h \big( \mu X- (\mu^2-2)Y \big) + (\mu^2-2)X -  (\mu^3-3\mu)Y.
\end{eqnarray*}
\end{proposition}

\begin{proof}
Let $\lambda_{\pm} = (\mu \pm \sqrt{\mu^2-4})/2$ be the eigenvalues of $M$. Note that $\lambda_+ \lambda_- = 1$ and $\lambda_+ + \lambda_- = \mu$. Let
$$P = \left[ \begin{array}{cc}
f & f \\
h-\lambda_+ & h-\lambda_- \end{array} \right].$$
Then $M = P \text{diag}(\lambda_-, \lambda_+) P^{-1}$.

Let $\alpha = h-\lambda_+$ and $\beta = h-\lambda_-$. Note that $\alpha \beta = h^2-h(\lambda_+ + \lambda_-)+1 = h^2-h(h+e)+1=-fg$ and $\alpha + \beta = 2h -(\lambda_+ + \lambda_-)= h-e$. With respect to the basis $\{E,H,F\}$ of $sl_2(\BC)$, the matrix of the adjoint action of $P$ is
$$Ad_P = \frac{1}{\lambda_- - \lambda_+} \left[ \begin{array}{ccc}
-f & 2f & f \\
\alpha & -(\alpha + \beta) & -\beta \\
\alpha^2/f & -2\alpha\beta/f & -\beta^2/f\end{array} \right].$$
Then $Ad_{M} = (Ad_P) \text{diag}(\lambda_-^2, 1, \lambda_+^2) (Ad_P)^{-1}.$ Hence
\begin{eqnarray*}
\sum_{i=0}^{n-1}(Ad_M)^i &=& (Ad_P) \text{diag}(\lambda_-^{n-1} S_{n-1}(\mu), n, \lambda_+^{n-1} S_{n-1}(\mu)) (Ad_P)^{-1}\\
&=& \frac{1}{(\lambda_- - \lambda_+)^2} \left[ \begin{array}{ccc}
C_{11} & C_{12} & C_{13} \\
C_{21} & C_{22} & C_{23} \\
C_{31} & C_{32} & C_{33}\end{array} \right],
\end{eqnarray*}
where 
\begin{eqnarray*}
C_{11} &=& -2 n \alpha\beta + (\beta^2\lambda_-^{n-1} + \alpha^2 \lambda_+^{n-1}) S_{n-1}(\mu)\\
C_{12} &=& 2f \big( -n(\alpha + \beta) + (\beta \lambda_-^{n-1} + \alpha \lambda_+^{n-1}) S_{n-1}(\mu) \big)\\
C_{13} &=& f^2 (2n - (\lambda_-^{n-1} + \lambda_+^{n-1}) S_{n-1}(\mu)),\\
C_{21} &=& (\alpha\beta / f) \big( n (\alpha + \beta) - (\beta \lambda_-^{n-1} + \alpha \lambda_+^{n-1}) S_{n-1}(\mu) \big)\\
C_{22} &=& n(\alpha + \beta)^2 - 2 \alpha \beta (\lambda_-^{n-1} + \lambda_+^{n-1}) S_{n-1}(\mu),\\
C_{23} &=& f \big( -n(\alpha + \beta) + (\alpha \lambda_-^{n-1} + \beta \lambda_+^{n-1}) S_{n-1}(\mu) \big)\\
C_{31} &=& (\alpha\beta / f)^2 (2n - (\lambda_-^{n-1} + \lambda_+^{n-1}) S_{n-1}(\mu)),\\
C_{32} &=& (2\alpha\beta /f) \big( n(\alpha + \beta) - (\alpha \lambda_-^{n-1} + \beta \lambda_+^{n-1}) S_{n-1}(\mu) \big)\\
C_{33} &=& -2n\alpha\beta + (\alpha^2 \lambda_-^{n-1} + \beta^2 \lambda_+^{n-1}) S_{n-1}(\mu).
\end{eqnarray*}
The proposition then follows from $(\lambda_- - \lambda_+)^2 = \mu^2-4$, $\alpha \beta =- fg$, $\alpha + \beta = h-e$,
\begin{eqnarray*}
\alpha \lambda_-^{n-1} + \beta \lambda_+^{n-1} &=& h (\lambda_-^{n-1} + \lambda_+^{n-1}) -  (\lambda_-^{n-2} + \lambda_+^{n-2}),\\
\beta \lambda_-^{n-1} + \alpha \lambda_+^{n-1} &=& h ( \lambda_-^{n-1} + \lambda_+^{n-1}) -  (\lambda_-^{n} + \lambda_+^{n}),\\
\alpha^2 \lambda_-^{n-1} + \beta^2 \lambda_+^{n-1} &=&  h^2 (\lambda_-^{n-1} + \lambda_+^{n-1}) - 2h (\lambda_-^{n-2} + \lambda_+^{n-2}) + (\lambda_-^{n-3} + \lambda_+^{n-3}),\\
\beta^2\lambda_-^{n-1} + \alpha^2 \lambda_+^{n-1} &=&  h^2(\lambda_-^{n-1} + \lambda_+^{n-1}) - 2h (\lambda_-^{n} + \lambda_+^{n}) + (\lambda_-^{n+1} + \lambda_+^{n+1}),
\end{eqnarray*}
and $\lambda_-^{n-l} + \lambda_+^{n-l} = (\lambda_-^{l-1} + \lambda_+^{l-1}) S_{n-1}(\mu) - (\lambda_-^{l} + \lambda_+^{l}) S_{n-2}(\mu).$
\end{proof}

\begin{proposition} \label{prop2}
Suppose $M = \left[ \begin{array}{cc}
e & f \\
g & h \end{array} \right] \in SL_2(\BC).$ 
Let $\mu = \tr M$. Then 
$$
(Ad_M)^n = \left[ \begin{array}{ccc}
D_{11} & D_{12} & D_{13} \\
D_{21} & D_{22} & D_{23} \\
D_{31} & D_{32} & D_{33}\end{array} \right],
$$
where 
\begin{eqnarray*}
D_{11} &=& \big( S_{p}(\mu) - hS_{p-1}(\mu) \big)^2,\\
D_{12} &=& -2fS_{p-1}(\mu) \big( S_{p}(\mu) - hS_{p-1}(\mu) \big),\\
D_{13} &=& -f^2 S^2_{p-1}(\mu),\\
D_{21} &=& - g S_{p-1}(\mu) \big( S_{p}(\mu) - hS_{p-1}(\mu) \big),\\
D_{22} &=& \big( S_{p}(\mu) - hS_{p-1}(\mu) \big) \big( S_{p}(\mu) - eS_{p-1}(\mu) \big) + fg S^2_{p-1}(\mu),\\
D_{23} &=& fS_{p-1}(\mu) \big( S_{p}(\mu) - eS_{p-1}(\mu) \big),\\
D_{31} &=& -g^2 S^2_{p-1}(\mu),\\
D_{32} &=& 2gS_{p-1}(\mu) \big( S_{p}(\mu) - eS_{p-1}(\mu) \big),\\
D_{33} &=& \big( S_{p}(\mu) - eS_{p-1}(\mu) \big)^2.
\end{eqnarray*}
\end{proposition}

\begin{proof}
The proof of Proposition \ref{prop2} is similar to that of Proposition \ref{prop1}.
\end{proof}

\section{Proof of Theorem \ref{thm-atp}}

\label{proof}

Recall that $K=J(2m,2n)$ and $\pi_1(K) = \la a,b~|~w^na=bw^n \ra$, where $a,b$ are meridians and $w=(ba^{-1})^m(b^{-1}a)^m$. 
 Suppose $\rho: \pi_1(K) \to SL_2(\BC)$ is a nonabelian representation. Up to conjugation, we may assume that $$\rho(a) = \left[ \begin{array}{cc}
s & 1 \\
0 & s^{-1} \end{array} \right] \quad \text{and} \quad \rho(b) = \left[ \begin{array}{cc}
s & 0 \\
2-y & s^{-1} \end{array} \right]$$ where $s \not=0$ and $y \not= 2$ satisfy the Riley equation $\phi_K(s,y)=0$.

The formulas in the following proposition are taken from \cite{Tr1}.

\begin{proposition} \label{Riley}
$i)$ $\rho(w)=\left[ \begin{array}{cc}
w_{11} & w_{12} \\
(2-y)w_{12} & w_{22} \end{array} \right]$, where
\begin{eqnarray*}
w_{11} &=& S^2_{m}(y) + (2-2y)S_m(y)S_{m-1}(y) + (1+2s^2-2y-s^2y+y^2)S^2_{m-1}(y),\\
w_{12} &=& (s^{-1}-s)S_m(y)S_{m-1}(y) + (s^{-1}+s-s^{-1}y)S^2_{m-1}(y),\\
w_{22} &=& S^2_{m}(y) - 2S_m(y)S_{m-1}(y) + (1+2s^{-2}-s^{-2}y)S^2_{m-1}(y).
\end{eqnarray*}
Hence $z:= \tr \rho(w)=2S^2_{m}(y) - 2yS_m(y)S_{m-1}(y) + \big(  (2+2s^2+2s^{-2})(1-y) + y^2 \big) S^2_{m-1}(y)$.

$ii)$ $\phi_K(s,y) = S_{n-2}(z) - \big[ 1  - (y-s^2-s^{-2})S_{m-1}(y) \big( S_{m-1}(y)- S_{m-2}(y) \big) \big] S_{n-1}(z)$.

$iii)$ $S^2_{n-1}(z) = \big[ (y-s^2-s^{-2})S^2_{m-1}(y) \left( 2-s^2-s^{-2} + (y-s^2-s^{-2})(y-2)S^2_{m-1}(y) \right) \big]^{-1}$.
\end{proposition}

Let $r=w^naw^{-n}b^{-1}$. We have $\Delta_{K}^{Ad \circ \rho}(t)=\det \Phi \left( \frac{\partial r}{\partial a} \right) \big/ \det \Phi(b-1)$. It is easy to see that $\det \Phi(b-1) = (t-1)(t-s^2)(t-s^{-2})$.

For an integer $p$ and a word $u$ (in 2 letters $a,b$), let $\delta_p(u)=1+u+\cdots+u^{p}$. 

\begin{lemma} \label{r/a}
We have 
$$\frac{\partial r}{\partial a} = w^n \big[ 1+(1-a)\delta_{n-1}(w^{-1})  (a^{-1}b)^m (b^{-1} -1) \delta_{m-1}(ab^{-1}) \big].$$
\end{lemma}

\begin{proof}
The lemma follows from a direct calculation.
\end{proof}

\begin{proposition} \label{s1}
We have
$$\Phi(\delta_{n-1}(w^{-1})) = \frac{1}{z^2-4}\left[ \begin{array}{ccc}
E_{11} & E_{12} & E_{13} \\
E_{21} & E_{22} & E_{23} \\
E_{31} & E_{32} & E_{33}\end{array} \right],$$ 
where
\begin{eqnarray*}
z &=& 2S^2_{m}(y) - 2yS_m(y)S_{m-1}(y) + \big(  (2+2s^2+2s^{-2})(1-y) + y^2 \big) S^2_{m-1}(y),\\
E_{11} &=&  -2n (y-2)w_{12}^2 + w_{11}^2 (2X-z Y) -  2w_{11}(z X -  2 Y) + (z^2-2)X - z Y,\\
E_{12} &=& -2w_{12} \big( n(w_{22}-w_{11}) + w_{11} (2X -z Y) - z X + 2 Y \big),\\
E_{13} &=& w_{12}^2 (2n- 2 X + z Y),\\
E_{21} &=& -(y-2)w_{12} \big( n(w_{11}-w_{22}) - w_{11} (2X -z Y) + z X - 2 Y \big),\\
E_{22} &=& n(w_{22}-w_{11})^2-2(y-2)w^2_{12}(2X - z Y),\\
E_{23} &=& -w_{12} \big(n(w_{22}-w_{11}) + w_{11} \big( 2X - z Y) - z X + (z^2-2)Y \big),\\
E_{31} &=& (y-2)^2w_{12}^2 (2n- 2X + z Y),\\
E_{32} &=& -2(y-2)w_{12} \big( n(w_{11}-w_{22})-w_{11}(2X-z Y) + z X - (z^2-2)Y \big),\\
E_{33} &=& -2n (y-2)w_{12}^2 + w_{11}^2(2X - z Y) - 2 w_{11} (z X- (z^2-2)Y) \\
&& + \, (z^2-2)X -  (z^3-3z)Y,\\
X &=& \big[ (y-s^2-s^{-2})S^2_{m-1}(y) \left( 2-s^2-s^{-2} + (y-s^2-s^{-2})(y-2)S^2_{m-1}(y) \right) \big]^{-1},\\
Y &=&  \big[ 1  - (y-s^2-s^{-2})S_{m-1}(y) \big( S_m(y) - (y-1)S_{m-1}(y) \big) \big] X.
\end{eqnarray*}
\end{proposition}

\begin{proof}
Recall that $\rho(w^{-1})=\left[ \begin{array}{cc}
w_{22} & -w_{12} \\
(y-2)w_{12} & w_{11} \end{array} \right]$, where $w_{ij}$ are computed in Proposition \ref{Riley}(i), and $z=\tr \rho(w^{-1})$. Since $\phi_K(s,y)=0$, by Proposition \ref{Riley}(ii) we have
$$S_{n-1}(z)S_{n-2}(z) =  \big[ 1  - (y-s^2-s^{-2})S_{m-1}(y) \big( S_m(y) - (y-1)S_{m-1}(y) \big) \big] S^2_{n-1}(z).$$

By Proposition \ref{Riley}(iii) we have $$S^2_{n-1}(z) = \big[ (y-s^2-s^{-2})S^2_{m-1}(y) \left( 2-s^2-s^{-2} + (y-s^2-s^{-2})(y-2)S^2_{m-1}(y) \right) \big]^{-1}.$$
The proposition then follows by applying Proposition \ref{prop1}.
\end{proof}

\begin{proposition} \label{s2}
We have 
$$
\Phi((a^{-1}b)^m) = \left[ \begin{array}{ccc}
F_{11} & F_{12} & F_{13} \\
F_{21} & F_{22} & F_{23} \\
F_{31} & F_{32} & F_{33}\end{array} \right],
$$
where 
\begin{eqnarray*}
F_{11} &=& \big( S_{m}(y) - S_{m-1}(y) \big)^2,\\
F_{12} &=& 2s^{-1}S_{m-1}(y) \big( S_{m}(y) - S_{m-1}(y) \big),\\
F_{13} &=& -s^{-2} S^2_{m-1}(y),\\
F_{21} &=& s(y-2) S_{m-1}(y) \big( S_{m}(y) - S_{m-1}(y),\\
F_{22} &=& \big( S_{m}(y) - S_{m-1}(y) \big) \big( S_{m}(y) - (y-1)S_{m-1}(y) \big) + (y-2) S^2_{m-1}(y),\\
F_{23} &=& -s^{-1}S_{m-1}(y) \big( S_{m}(y) - (y-1)S_{m-1}(y) \big),\\
F_{31} &=& -s^2(y-2)^2 S^2_{m-1}(y),\\
F_{32} &=& -2s(y-2)S_{m-1}(y) \big( S_{m}(y) - (y-1)S_{m-1}(y) \big),\\
F_{33} &=& \big( S_{m}(y) - (y-1)S_{m-1}(y) \big)^2.
\end{eqnarray*}
\end{proposition}

\begin{proof}
Note that $\rho(a^{-1}b) = \left[ \begin{array}{cc}
y-1 & -s^{-1} \\
-s(y-2) & 1 \end{array} \right]$ and $\tr \rho(a^{-1}b)=y$. The proposition follows by applying Proposition \ref{prop2}.
\end{proof}

\begin{proposition} \label{s3}
We have
$$\Phi(\delta_{m-1}(ab^{-1})) = \frac{1}{y^2-4}\left[ \begin{array}{ccc}
G_{11} & G_{12} & G_{13} \\
G_{21} & G_{22} & G_{23} \\
G_{31} & G_{32} & G_{33}\end{array} \right],$$ 
where
\begin{eqnarray*}
G_{11} &=& (y-2) \big[ 2m + 2 S_m(y) S_{m-1}(y) - y S^2_{m-1}(y) \big],\\
G_{12} &=& 2s(y-2) \big[ m + S_m(y) S_{m-1}(y) - (y+1) S^2_{m-1}(y) \big],\\
G_{13} &=& s^2 \big[ 2m -y S_m(y) S_{m-1}(y) + (y^2-2) S^2_{m-1}(y) \big],\\
G_{21} &=& s^{-1}(y-2)^2 \big[ m +  S_m(y) S_{m-1}(y) - (y+1) S^2_{m-1}(y) \big],\\
G_{22} &=& (y-2) \big[ (y-2)m + 2y S_m(y) S_{m-1}(y) - (2y^2-4) S^2_{m-1}(y) \big],\\
G_{23} &=& s(y-2) \big[ m -(y+1) S_m(y) S_{m-1}(y) +(y^2+y-1) S^2_{m-1}(y) \big],\\
G_{31} &=& s^{-2}(y-2)^2 \big[ 2m -y S_m(y) S_{m-1}(y) + (y^2-2) S^2_{m-1}(y) \big],\\
G_{32} &=& 2s^{-1}(y-2)^2 \big[ m - (y+1) S_m(y) S_{m-1}(y) + (y^2+y-1) S^2_{m-1}(y) \big],\\
G_{33} &=& (y-2) \big[ 2m + (y^2-2) S_m(y) S_{m-1}(y) - (y^3-3y) S^2_{m-1}(y) \big].
\end{eqnarray*}
\end{proposition}

\begin{proof}
Note that $\rho(ab^{-1}) = \left[ \begin{array}{cc}
y-1 & s \\
s^{-1}(y-2) & 1 \end{array} \right]$ and $S_{m-2}(y) = y S_{m-1}(y) - S_m(y)$. The proposition follows by applying Proposition \ref{prop1}.
\end{proof}

Since $\det \Phi(w)=1$, by Lemma \ref{r/a} we have
$$
\det \Phi(\frac{\partial r}{\partial a}) =  \det \Phi \Big( 1+(1-a)\delta_{n-1}(w^{-1})  (a^{-1}b)^m (b^{-1} -1) \delta_{m-1}(ab^{-1}) \Big).
$$

With the following formulas
\begin{eqnarray*} 
\Phi(1-a) &=& \left[ \begin{array}{ccc}
1-s^2t & 2st & t \\
0 & 1-t & - s^{-1}t \\
0 & 0 & 1- s^{-2}t\end{array} \right],\\
\Phi(\delta_{n-1}(w^{-1})) &=& \frac{1}{z^2-4}\left[ \begin{array}{ccc}
E_{11} & E_{12} & E_{13} \\
E_{21} & E_{22} & E_{23} \\
E_{31} & E_{32} & E_{33}\end{array} \right] \qquad (\text{by Proposition } \ref{s1}),\\
\Phi((a^{-1}b)^m) &=& \left[ \begin{array}{ccc}
F_{11} & F_{12} & F_{13} \\
F_{21} & F_{22} & F_{23} \\
F_{31} & F_{32} & F_{33}\end{array} \right] \qquad  \qquad \quad \, \, (\text{by Proposition } \ref{s2}),\\
\Phi(b^{-1} -1) &=& \left[ \begin{array}{ccc}
-1 + s^{-2}t^{-1} & 0 & 0 \\
(2-y)s^{-1}t^{-1} & -1+t^{-1} & 0 \\
-(y-2)^2t^{-1} & 2s(y-2)t^{-1} & -1 + s^{2}t^{-1}\end{array} \right],\\
\Phi(\delta_{m-1}(ab^{-1})) &=& \frac{1}{y^2-4}\left[ \begin{array}{ccc}
G_{11} & G_{12} & G_{13} \\
G_{21} & G_{22} & G_{23} \\
G_{31} & G_{32} & G_{33}\end{array} \right] \qquad (\text{by Proposition } \ref{s3}),
\end{eqnarray*} and the help of Mathematica, we obtain
\begin{eqnarray*}
\det \Phi(\frac{\partial r}{\partial a}) &=& \frac{(t-1)^2(t-s)(t-s^{-1})}{(y-s^2-s^{-2}) \big( 2-s^2-s^{-2}  + (y-2)(y-s^2-s^{-2}) S^2_{m-1}(y) \big)} \\
&& \times \left( mn t^2 - \frac{A'(s^4+s^{-4}) + B'(s^2+s^{-2})+C'}{4   + (y-2)(y-s^2-s^{-2}) S^2_{m-1}(y)} \, t + mn\right),
\end{eqnarray*}
where
\begin{eqnarray*}
A' &=& 2n \Big( \frac{2m +2 S_m(y)S_{m-1}(y) - y S^2_{m-1}(y)}{y+2} \Big),\\
B' &=& 8n \Big( \frac{2m +2 S_m(y)S_{m-1}(y) - y S^2_{m-1}(y)}{y+2} \Big) \\
&& - \, 4 \big( S_m(y) + S_{m-1}(y) \big)^2  \Big( \frac{m S_{m}(y) + (m+1) S_{m-1}(y)}{(y+2) S_{m-1}(y)} \Big)\\
&& + \, 6m S^2_m(y)+(4+4m-4n-2my) S_{m}(y) S_{m-1}(y)\\
&& + \,  (2+2m-4mn-y+2ny+2mny) S^2_{m-1}(y),\\
C' &=& 12n \Big( \frac{2m +2 S_m(y)S_{m-1}(y) - y S^2_{m-1}(y)}{y+2} \Big) \\
&& - \, 8 \big( S_m(y) + S_{m-1}(y) \big)^2  \Big( \frac{m S_{m}(y) + (m+1) S_{m-1}(y)}{(y+2) S_{m-1}(y)} \Big)\\
&& + \, 12m S^2_m(y)+(4+4m-4n-2my) S_{m}(y) S_{m-1}(y)\\
&& + \,  (4+4m-8mn-2y+4ny+4mny) S^2_{m-1}(y) \\
&& - \, (2mn-1)  \big( 2S_m(y) - yS_{m-1} (y) \big)^2.
\end{eqnarray*}
Here we should remark that we used the formula $S^2_m(y) - y S_m(y) S_{m-1}(y) + S^2_{m-1}(y)=1$ to greatly simplify the computations in Mathematica.

Since $s+s^{-1}=x$, $s^2+s^{-2}=x^2-2$ and $s^4+s^{-4}=x^4-4x^2+2$ we get
\begin{eqnarray*}
\Delta_K^{Ad \circ \rho}(t) &=& \frac{t-1}{(y+2-x^2) \big( 4 - x^2  + (y-2)(y+2-x^2) S^2_{m-1}(y) \big)} \\
&& \times \left( mn t^2 - \frac{Ax^4 + Bx^2+C}{4   + (y-2)(y+2-x^2) S^2_{m-1}(y)} \, t + mn\right),
\end{eqnarray*}
where $A=A'$, $B=B'-4A'$ and $C=2A'-2B'+C'$. 

This completes the proof of Theorem \ref{thm-atp}.

\section{Acknowlegements} This work was partially supported by a grant from the Simons Foundation (\#354595 to Anh Tran).

\end{document}